\documentclass[12pt,a4paper]{amsart}
\usepackage{amsmath,amscd}
\usepackage{latexsym,amsmath,amssymb,mathrsfs,setspace,enumerate,tikz}
\usepackage[T1]{fontenc}
\usepackage[utf8]{inputenc}
\usepackage{lmodern}
\usepackage{amssymb}
\usepackage{tikz}
\usetikzlibrary{automata,positioning}

\newtheorem{thm}{Theorem}

\newtheorem{exam}{Example}

\newtheorem{dfn}{Definition}

\title{\textbf{GRAPH INVERSE SEMIGROUPS AND THEIR SUBSTRUCTURES}}
\author {P. G. Romeo$^1$ and Alanka Thomas$^2$}
\address{Dept. of Mathematics, Cochin University of Science and Technology, Kochi, Kerala, INDIA.}
\email{$romeo_-parackal@yahoo.com,\, alankathomas1@gmail.com $}
\begin{document}

\begin{abstract}
In this paper we discuss graph inverse semigroups $I(\Gamma)$ which are constucted from a directed graph $\Gamma$ and study several interesting properties of graph inverse semigroups such as the nature of its idempotents, the structure of semilattice of idempotents and the like. A necessary and sufficient condition on a directed graph $\Gamma$ to have the corresponding $I(\Gamma)$ a primitive inverse semigroup is provided. Further the general form of elements in a local submonoids of $I(\Gamma)$ and the Green's relations in graph inverse semigroups in graph theoretic terms are also described.

\end{abstract}
\maketitle
\section{Introduction}
Graph inverse semigroups are inverse semigroups constructed from a directed graph. 
Given a directed graph $\Gamma$ the graph inverse semigroup, which we dente by $I(\Gamma)$ is an inverse semigroup and being an inverse semigroup the set of idempotents $E(I(\Gamma))$ form a semilattice under natural partial ordering. Structure of the semilattice $E(I(\Gamma))$ of few digraphs are discussed in section $3$. $I(\Gamma)$ is primitive if and only if corresponding $\Gamma$ is totally disconnected or trivial. General form of an element in a local submonoid at some idempotents in $I(\Gamma)$ is provided. 
and Theorem $(4)$ gives a necessary and sufficient condition on $\Gamma$ to have local submonoid at each vertex to be trivial. In subsection $(3.3)$ we discuss Green's relations in graph inverse semigroups. 

\section{Preliminaries}

A  set $S$ equipped  with an associative binary operation $*$ is a semigroup and is denoted by $(S,*)$. An element $e\,\in\, S$ is an idempotent if $e^2\,=\,e$, the set of all idempotents in $S$ is denoted by 
$E(S)$. An element $a\,\in \,S$ is regular if there exists $b\,\in\,S$ such that $aba\,=a$. If each element of $S$ is regular, then $S$ is called a regular semigroup. If $a\,\in\,S$ is regular, we can find $a'\,\in \, S$ such that $aa'a\,=\,a$ and $a'aa'\,=\,a'$, such an $a'$ is called an inverse of $a$. Note that an 
element in a regular semigroup can have more than one inverse. If $a'$ is an inverse of $a\,\in\,S$, then $aa'$ and $a'a$ are idempotents in $S$.

\begin{dfn}
An inverse semigroup is a regular semigroup, in which every element $a$ has a unique inverse.
\end{dfn}
\begin{thm}
Let $S$ be a regular semigroup. Each element in $S$ has a unique inverse if and only if idempotents is $S$ commutes. ie., an inverse semigroup is a regular semigroup in which idempotents commutes.
\end{thm}

\begin{dfn} An ordering $\leq$ defined by $a \, \leq \, b\Leftrightarrow\, \text{there exists} \,\,e,f\,\in \,
E(S)$ such that $a\,=be\,=fb$, called the natural partial ordering on $S$. If $S$ is an inverse semigroup, $\leq$ when restricted to $E(S)$ form a semilattice.
\end{dfn}

\begin{dfn}
Let $S$ be a semigroup,binary relations $\mathscr{L},\mathscr{R},\mathscr{H},\mathscr{D},\mathscr{J}$ are called Green's relations and are defined as follow.
\begin{enumerate}
\item $a\,\mathscr{L}\,b$ if and only if $S^1a=S^1b$, that is, there exists  $s,\,s'\,\in S^1$ such that $sa=b$ and $s'b=a$.
\item $a\,\mathscr{R}\,b$ if and only if $aS^1=bS^1$, that is, there exixtx $s,\,s'\,\in S^1$ such that $as=b$ and $bs'=a$.
\item $a\,\mathscr{J}\,b$ if and only if $S^1aS^1=S^1bS^1$, that is, there exixtx $s,\,s',\,t,\,t'\in S^1$ such that $sat=b$ and $s'bt'=a$.
\item $\mathscr{D}\,=\,\mathscr{L}\circ\mathscr{R}\,=\,\mathscr{R}\circ\mathscr{L}$
\item $\mathscr{H}\,=\,\mathscr{L}\cap\mathscr{R}$.
\end{enumerate}
\end{dfn}
The theorem below follows easily.
\begin{thm}
Let $S$ be an inverse semigroup with semilattice of idempotents $E$,then
\begin{enumerate}
\item $a\,\mathscr{R}\,b\,$ if and only if $ \, aa^{-1}\,=\,bb^{-1}$.
\item $a\,\mathscr{L}\,b\,$ if and only if $ \, a^{-1}a\,=\,b^{-1}b$.
\item $a\,\mathscr{H}\,b\,$ if and only if $ \, a^{-1}a\,=\,b^{-1}b$ and $aa^{-1}\,=\,bb^{-1}$.
\end{enumerate}
\end{thm}

 Next we recall some basic definitions and results from graph theory needed in the sequel.

A directed graph $\Gamma$ consists of a non-empty set of vertices $\Gamma^0$, a possibly empty set of edges $\Gamma^1$, source map $s\,:\,\Gamma^1\rightarrow\Gamma^0$ which maps each edge $e$ to its source $s(e)$ and range map $r\,:\,\Gamma^1\rightarrow\Gamma^0$ which maps each edge $e$ to its range $r(e)$. An edge $e$ is denoted as an arrow $s(e)\xrightarrow{e} r(e)$. If $s(e)= r(e)$ then $e$ is called a loop. 
$\mid s^{-1}(v)\mid$ gives the number of edges starting from $v$, which is the outdegree of $v$. Similarly 
$\mid r^{-1}(v)\mid$ gives the indegree of $v$. A vertex having zero outdegree is called a sink and a vertex with zero indegree is called a source.
\begin{dfn}
A directed path in $\Gamma$ is a finite seguence $p=e_1e_2...e_n$ where $e_i\,\in\,\Gamma^1$ with $r(e_i)=s(e_{i+1})$ forall $i=1,2,...{n-1}$. Note that $s(p)=s(e_1)$ and $r(p)= r(e_n)$ and we say $p$ is a directed path from $s(p)$ to $r(p)$.
\end{dfn}

A vertex $v$ may also regarded as a directed path (rivial path) with $s(v)=r(v)=v$r. For each $e$ there is an inverse edge $e^*$ with $s(e^*)=r(e)$ and $r(e^*)=s(e)$. These edges are also called ghost edges. $(\Gamma^1)^*$ denote the set of all ghost edges of $\Gamma$. If $p=e_1e_2...e_n$ is a directed path in $\Gamma$, its inverse path is defined as $p^*={e_n}^*{e_{n-1}}^*...{e_1}^*$.
\begin{dfn}
A directed graph $\Gamma$ is said to be connected if forall vertices $u$ and $v$ there exists atleast one directed path from $u$ to $v$ or $v$ to $u$. Two vertices $u$ and $v$ are said to be strongly connected if there exist directed paths from $u$ to $v$ and $v$ to $u$. 
\end{dfn}

Clearly strongly connectedness is an equivalence relation on $\Gamma^0$  and this describes the digraph into strongly connected components.

\section{Graph inverse semigroups}
 Next we describe an inverse semigroup which is generated by a directed  which we call the graph inverse semigroup and discuses  various properties of garph inverse semigroups.
\begin{dfn}(cf.\cite{meakin}) Let $\Gamma$ be a directed graph. The graph inverse semigroup denoted by $I (\Gamma)$ is a semigroup generated by $W\,=\,\Gamma^0\cup \Gamma^1\cup {\Gamma^1}^*$ subject to the relations,
 \begin{enumerate}
 \item $s(e)e\,= e\,=\, er(e)\quad \forall \,e\, \in \Gamma^0\cup \Gamma^1\cup {\Gamma^1}^*$
 \item $uv\, =\, \delta_{uv}\,u \quad \forall \, u,\,v\,\in \Gamma^0$
 \item $e^* f\, = \delta_{ef}\, r(e)\quad \forall\, e,\, f\, \in \Gamma^1 $
 \end{enumerate}
 \end{dfn}
It is easy to see that $I(\Gamma)$ is an inverse semigroup (cf.\cite{alan}). $\Gamma$ be a directed graph and $p,\,q$ be directed paths in $\Gamma$ with same range, that is $r(p)\,=\,r(q)$, these paths provides an ordered pair $(p,q)$. Let $T$ be the set of all ordered pairs $(p,q)$ where $p$ and $q$ are directed paths with $r(p)\, =\, r(q)$ along with zero. Define a binary operation $*$ on $T$ by,\\
\[ (p_1,q_1)\,*\,(p_2,q_2)\, = \left\{ \begin{array}{ll}
         (p_1t,q_2) & \mbox{if $p_2\,=\,q_t\,for\,some \,directed \,path \,t$};\\
        (p_1,q_2t) & \mbox{if $q_1\,=p_2t\,for\,some\,directed\,path\,t$};\\
        0 & \mbox{otherwise}.
        \end{array} \right. \] 
It is easy to observe that $*$ is a closed binary operation. For $(p_1,q_1),\,(p_2,q_2)$ and $(p_3,q_3)$ in $T$ such that $q_1$ or $p_2$ are not initial component of the other or $q_2$ or $p_3$ are not initial component of the other then $[(p_1,q_1)\,*\,(p_2,q_2)]\,*\,(p_3,q_3)$ and $(p_1,q_1)\,*\,[(p_2,q_2)\,*\,(p_3,q_3)]$ are $0$. On the otherhand, for 
$q_1\,=\,p_2t_1$ and $q_2\,=\, p_3t_2$, then\\
\[ \begin{array}{lcl}
\mbox{$[(p_1,q_1)\,*\,(p_2,q_2)]\,*\,(p_3,q_3)$} & = & [(p_1,p_2t_1)\,*\,(p_2,p_3t_2)]\,*\,(p_3,q_3) \\
\mbox{} & = & (p_1,p_3t_2t_1)\,*\,(p_3,q_3) \\
\mbox{} & = & (p_1,q_3t_2t_1)
\end{array}\]
and
\[ \begin{array}{lcl}
\mbox{$(p_1,q_1)\,*\,[(p_2,q_2)\,*\,(p_3,q_3)]$} & = & (p_1,p_2t_1)\,*\,[(p_2,p_3t_2)\,*\,(p_3,q_3)] \\
\mbox{} & = & (p_1,p_2t_1)\,*\,(p_2,q_3t_2) \\
\mbox{} & = & (p_1,q_3t_2t_1)
\end{array}\]
ie.,  $*$ is associative. Similarly prove associativity in other cases as well.  Thus $(T,*)$ is a semigroup with zero. Let $(p,q)\,\in \, T$ and $(p',q')$ be its inverse, ie., 
$$(p,q)*(p',q')*(p,q)\,=\,(p,q)\,\text{ and}\,(p',q')*(p,q)*(p',q')\,=\,(p',q')$$
since both products above are nonzero there exists directed paths $t_1$ and $t_2$ such that $p'\,=\,qt_1$ and $q'\,=\,pt_2$ and so  
\[
\begin{array}{lll}
\mbox{} & (p,q)*(p',q')*(p,q)\,& =\,(p,q)\\
\mbox{} & (p,q)(qt_1,pt_2)(p,q) & =\,(p,q) \\
\mbox{} & (pt_1,qt_2) & =\,(p,q)
\end{array}
\]
which implies that $t_1$ and $t_2$ are trivial paths. Hence $(q,p)$ is the unique inverse of $(p,q)$ 
ie., $T$ is an inverse semigroup.

\par Define a map $\alpha\, :\,W\,\rightarrow\,T$ as follows,
\[ \begin{array}{lcl}
\mbox{$\alpha(v)$} & = & (v,v) \\
\mbox{$\alpha(e)$} & = & (e,r(e)) \\
\mbox{$\alpha(e^*)$} & = & (r(e),e)\\
\mbox{$\alpha(0)$} & = & 0
\end{array}\]
For the free semigroup $F_W$ generated by $W$, by the characteristic property of free semigroup it is easy to see that, $\alpha$ can be extended to a unique homomorphism $\bar{\alpha}\,:\,F_W\rightarrow \, T$. Let $(p,q)$ be a non-zero element in $T$ where $p\,=\,a_1a_2...a_n$ , $q\,=\,b_1b_2...b_m$ and $a_i,\,b_j\, \in \Gamma^1 \, \forall \, i,\,j$, then,
\[ \begin{array}{lcl}
\mbox{$\bar{\alpha}(pq^*)$} & = & \bar{\alpha}(a_1a_2...a_n{b_m}^*{b_{m-1}}^*...{b_1}^*) \\
\mbox{} & = & \bar{\alpha}(a_1)*\bar{\alpha}(a_2)*...*\bar{\alpha}(a_n)*\bar{\alpha}({b_{m}}^*)*\bar{\alpha}({b_{m-1}}^*)*...*\bar{\alpha}({b_1}^*)\\
\mbox{} & = & (a_1,r(a_1))*...*(a_n,r(a_n))*(r(b_m),b_m)*...*(r(b_1),b_1)\\
\mbox{} & = & (a_1a_2...a_n\,,\,b_1b_2...b_m)\\
\mbox{} & = & (p,q)
\end{array}\]
ie., each element in $T$ has a preimage under $\bar{\alpha}$ which implies $\bar{\alpha}$ is onto. 

\par Let $\rho$ be the congruence on $F_W$ generated by the three conditions used to define a graph inverse semigroup, then the $\rho$-classes in $F_W$ have the same image  
under $\bar{\alpha}$. For, it will suffices to show that $\bar{\alpha}(W)$ satisfy the following 

\begin{enumerate}
\item \[\begin{array}{lcl}
\mbox{$\bar{\alpha}(u)\bar{\alpha}(v)$} & = & (u,u)*(v,v)\\
\mbox{} & = & {\left\{ \begin{array}{ll}
         (u,u) & \mbox{if $u=v$};\\
          0 & \mbox{Otherwise}
          \end{array} \right. }\\
\mbox{} & = & {\delta_{uv}(u,u)}\\
\mbox{} & = & \delta_{\bar{\alpha}(u)\bar{\alpha}(v)}\bar{\alpha}(u)\quad\forall\, u,\,v\, \in \Gamma^0
\end{array} 
\]
\item \[\begin{array}{lcl}
\mbox{$\bar{\alpha}(e^*)\bar{\alpha}(f)$} & = & (r(e),e)*(f,r(f))\\
\mbox{} & = & {\left\{ \begin{array}{ll}
         (r(e),r(e)) & \mbox{if $e=f$};\\
          0 & \mbox{Otherwise}
          \end{array} \right. }\\
\mbox{} & = & {\left\{ \begin{array}{ll}
         \bar{\alpha}(r(e)) & \mbox{if $e=f$};\\
          0 & \mbox{Otherwise}
          \end{array} \right. }\\
\mbox{} & = & \delta_{\bar{\alpha}(e)\bar{\alpha}(f)}\bar{\alpha}(r(e))\quad\forall\, e,\,f\, \in \Gamma^1
\end{array} 
\]
\item \[\begin{array}{lcl}
\mbox{$\bar{\alpha}(s(e))\bar{\alpha}(e)$} & = & (s(e),s(e))*(e,r(e))\\
\mbox{} & = & (s(e)e,r(e)) \\
\mbox{} & = & (e,r(e))\\
\mbox{} & = &  \bar{\alpha}(e)\quad\forall\, e\,\in \Gamma^0\\
\mbox{$\bar{\alpha}(s(e))\bar{\alpha}(e)$} & = & (e,r(e))(r(e),r(e))\\
\mbox{} & = & (e, r(e))\\
\mbox{} & = & \bar{\alpha}(e)
\end{array}
\]

ie., $\bar{\alpha}(s(e))\bar{\alpha}(e)\,=\,\bar{\alpha}(s(e))\bar{\alpha}(e)\,=\,\bar{\alpha}(e)\quad \forall\, e\, \in \,\Gamma^0.$
\end{enumerate}
Thus $\rho$ related elements have same image under $\bar{\alpha}$. Define a map $\phi\,:I(\Gamma)\, =\,F_W/\rho \rightarrow T$ by $\phi([t])\,=\,\bar{\alpha}(t)$, where $[t]$ is the $\rho$- congruence class containing $t$. Clearly $\phi$ is a surjective homomorphism. For $\phi([t])\, = \, \phi([s])$, we can find $pq^*\,\in\,[s]$ and $xy^*\,\in\,[t]$. Then 
$$\phi([pq^*])\,=\,\phi([s])\,=\,\phi([t])\,=\,\phi([xy^*])$$ which implies $\bar{\alpha}(pq^*)\, = \,\bar{\alpha}(xy^*)$. ie., $(p,q) \, = \, (x,y)$ and so $p = q$ and $r = s$. Hence $ pq^*\,= \, xy^*$. ie., 
$[s]\, =\, [t]$,and so $\phi$ is one-one. Therefore 
$\phi\, :\, I(\Gamma)\rightarrow T$ is an isomorphism and $I(\Gamma)$ is an inverse semigroup. Also, each element of $I(\Gamma)$ is of the form $pq^*$ for some directed paths $p$ and $q$ with $r(p)\, =\, r(q)$ and so each element of $I(\Gamma)$ are circutes   of the form $pq^*$ where $p$ and $q$ are directed paths in $\Gamma$ with same end points. Now we provide few examples of graph inverse semigroups.
\begin{exam}
Consider the directed graph $\Gamma$\\
\[
\begin{tikzpicture}[>=stealth,
   shorten >=1pt,
   node distance=2cm,
   on grid,
   auto,
   every state/.style={draw=black!60, fill=black!5, very thick}
  ]
\node (mid)        {u};
\node (right)[right=of mid] {v};
\path[->]
%   FROM       BEND/LOOP           POSITION OF LABEL   LABEL   TO
   (mid)  edge             node                      {e} (right)
   ;
\end{tikzpicture}
\]
then the graph inverse semigroup is $I(\Gamma)\,=\,\{u,\,v,\,e,\,e^*,\,ee^*,0$
\end{exam}

\begin{exam}
Consider the directed graph $\Gamma$ with single vertex and a loop on it.
\[
\begin{tikzpicture}
\node (mid)        {v};
\path[->]
%   FROM       BEND/LOOP           POSITION OF LABEL   LABEL   TO
   (mid)  edge [loop above]            node                      {e} (mid)  ;
\end{tikzpicture}
\]
the three conditions to define $I(\Gamma)$ can be rewritten as follows
\begin{enumerate}
\item $v^2\, =\, v$
\item $e^*e\,=\,v$
\item $ve\,=\,ev\,=\,e\,\text{and}\,\, ve^*\,=\,e^*v\,=\,e^*$
\end{enumerate}
ie., $v$ act as an identity in $I(\Gamma)$ and $I(\Gamma)$ is the monoid generated by the symbols $e,\,e^*$ subject to the condition that $e^*e\,=\, v$. Thus $I(\Gamma)\,=\,<e,\,e^*\mid e*e\, =\,v>$, is the bicyclic monoid generated by $e$ (cf.\cite{meakin}).
\end{exam}

If $\Gamma$ is a directed graph with a single vertex and n loops, then corresponding $I(\Gamma)$ is the polycyclic monoid with $n$ symbols.

\begin{exam}
Let $\Gamma$ be $P_3$, a directed graph with three vertices as follows,
\[\begin{tikzpicture}[>=stealth,
   shorten >=1pt,
   node distance=2cm,
   on grid,
   auto,
   every state/.style={draw=black!60, fill=black!5, very thick}
  ]
\node (mid) {v};
\node (right)[right= of mid]  {w};
\node (left)[left= of mid]{u};
\path[->]
(left) edge node {e} (mid)
(mid)edge node {f}(right);
\end{tikzpicture}
\]
\\ then\\
$I(\Gamma)=\{u,v,w,e,f,e^*,f^*,ee^*,ff^*,ef,(ef)^*,(ef)(ef)^*,eff^*,f(ef)^*,0\}$
\end{exam}

\subsection{Idempotents in $I(\Gamma)$ }
The element in $I(\Gamma)$ for any directed graph $\Gamma$ is of the form $pq^*$ where $p$ and $q$ are directed paths with same end point. 
If $q\,=\,pt$ for some directed path $t$, then $pq^*pq^*\,=\,p(pt)^*pq^*\,=\,p(qt)^*$ which implies that 
$q\,=\,qt$, ie., $t$ is a trivial path and so $p\,=\,q$
Thus nonzero idempotents in $I(\Gamma)$ are of the form $pp^*$ for some directed path. Since $I(\Gamma)$ is an inverse semigroup, its set of idempotents $E(I(\Gamma))$ form a semilattice under the natural partial 
ordering $e\,\leq\, f\,\Leftrightarrow\, e\,=\,ef\,=fe$. For $pp^*,\,qq^* \in E(I(\Gamma))$, if 
$p\,=\,qt$ for some directed path $t$, then 
$pp^*qq^*\,=\,p(qt)^*qq^*\,=\,p(qt)^*\,=\,pp^*$ which implies that $pp^*\,\leq\,qq^*$.

Conversely  $pp^*\,\leq\, qq^*$ if and only if $q$ is an initial component of $p$ (cf.\cite{meakin}). But $ 0\cdot pp^*\,=\, 0\, \forall \,pp^*\, \in E(I(\Gamma)),\,ie., 0$ is the minimum element in $E(I(\Gamma))$. Thus the 
maximal elements in $E(I(\Gamma))$ are $pp^*$ where $p$ is a directed path, for which there doesnot exist a path $q$ as its initial component, which is possible only if $p$ is the trivial path $v$ for some vertex $v$ in $\Gamma$. Thus maximal idempotents in $I(\Gamma)$ are precisely the vertices of $\Gamma$.

\begin{exam}
Consider \[
\begin{tikzpicture}
\node (mid)        {v};
\path[->]
%   FROM       BEND/LOOP           POSITION OF LABEL   LABEL   TO
   (mid)  edge [loop above]            node                      {e} (mid)  ;
\end{tikzpicture}
\]
then $I(\Gamma)$ is bicyclic semigroup $<e,\,e^*\mid e^*e\, =\,v>$ and 
$E(I(\Gamma))\,=\,\{0,v,e^n(e^n)^*\mid \,n\,\in \mathbb{N}\}$ maximum and minimum elements are $v$ and $0$  respectively. Further \\
\[
\begin{array}{lcl}
\mbox{$e^n(e^n)^*\,\leq\,e^m(e^m)^*$} & iff & e^m\, is\,an\,initial\,component \,of e^n\\
\mbox{} & iff & m\,\leq\,n\quad\forall\,m,n\,\in\mathbb{N}
\end{array}
\]
thus the lattice diagram of the semilattice $E(I(\Gamma))$ is an infinite chain having $v$ as the top most element and 0 on the bottom.
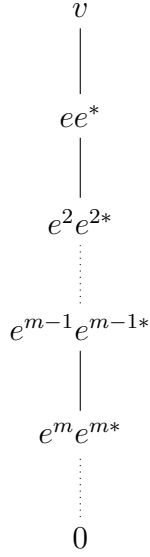
\[\begin{tikzpicture}[scale=.7]
  \node (one) at (0,2) {$v$};
  \node (a) at (0,0) {$ee^*$};
  \node (b) at (0,-2) {$e^2e^{2*}$};
  \node (c) at (0,-4) {$e^{m-1}e^{m-1*}$};
  \node (d) at (0,-6) {$e^me^{m*}$};
  \node (zero) at (0,-8) {$0$};
 \draw (d) -- (c) ;
\draw (b) -- (a) -- (one) ;
\draw[dotted] (c)--(b);
\draw[dotted] (zero)--(d);
\end{tikzpicture}
\]
\end{exam}

\begin{exam}
Consider $P_2$\[
\begin{tikzpicture}[>=stealth,
   shorten >=1pt,
   node distance=2cm,
   on grid,
   auto,
   every state/.style={draw=black!60, fill=black!5, very thick}
  ]
\node (mid)        {u};
\node (right)[right=of mid] {v};
\path[->]
%   FROM       BEND/LOOP           POSITION OF LABEL   LABEL   TO
   (mid)  edge             node                      {e} (right)
   ;
\end{tikzpicture}
\]
then $I(P_2)\,=\,\{u,\,v,\,e,\,e^*,\,ee^*,0\}$ and $E(I(P_2))\,=\,\{o,u,v,ee^*\}$ and the lattice diagram of $E(I(P_2))$ is given by,\\
\[\begin{tikzpicture}[scale=.7]
  \node (u) at (0,2) {$u$};
  \node (v) at (2,2) {$v$};
  \node (b) at (0,0) {$ee^*$};
  \node (zero) at (1,-2) {$0$};
 \draw (u) -- (b) ;
\draw (b) -- (zero) ;
\draw (v)--(zero);
\end{tikzpicture}
\]
\end{exam}

\begin{exam}
Consider $P_3$
\[
\begin{tikzpicture}[>=stealth,
   shorten >=1pt,
   node distance=2cm,
   on grid,
   auto,
   every state/.style={draw=black!60, fill=black!5, very thick}
  ]
\node (a) {$v_1$};
\node (b)[right=of a] {$v_2$};
\node (c)[right= of b] {$v_3$};
\path[->]
(a) edge node {$e_1$} (b)
(b) edge node {$e_2$} (c);
\end{tikzpicture}
\]
here
\[
\begin{split}
I(P_3)=\{v_1,v_2,v_3,e_1,e_2,{e_1}^*,{e_2}^*,{e_1}{e_1}^*,{e_2}{e_2}^*,e_1e_2,{e_1e_2}^*,{e_1e_2}(e_1e_2)^*,\\
e_2(e_1e_2)^*,e_1e_2{e_2}^*,0\}h
\end{split}
\]
$E(I(P_3))=\{0,v_1,v_2,v_3,e_1{e_1}^*,e_2{e_2}^*,(e_1e_2)(e_1e_2)^*\}$. 
Here the lattice diagram of $E(I(P_3))$ is as follows,\\
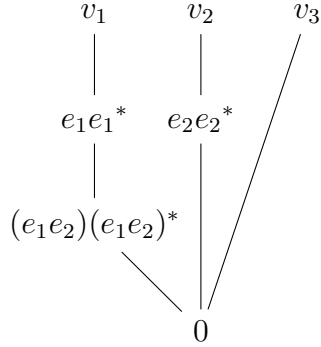
\[\begin{tikzpicture}[scale=.7]
  \node (u) at (0,2) {$v_1$};
  \node (v) at (2,2) {$v_2$};
  \node (w) at (4,2) {$v_3$};
\node (a) at (0,0) {$e_1{e_1}^*$};
\node (b) at (2,0) {$e_2{e_2}^*$};
\node (c) at (0,-2){$(e_1e_2)(e_1e_2)^*$};
  \node (zero) at (2,-4) {$0$};
 \draw (u) -- (a) ;
\draw (v) -- (b) ;
\draw (a)--(c);
\draw (c)--(zero);
\draw (b)--(zero);
\draw (w)--(zero);
\end{tikzpicture}
\]
\end{exam}

Thus we have for directed paths $P_2$ and $P_3$  lattice diagram has non-intersecting chains from each vertex ending at $0$. The chain at $v_i$  contains only idempotents $pp^*$ where $p$ is a path starting at $v_i$ and are arranged in ascending order of their length. Hence minimal nonzero idempotents are $pp^*$ where $p$ is the longest path starting from a vertex. In general for $P_n$,\\
\[
\begin{tikzpicture}[>=stealth,
   shorten >=1pt,
   node distance=2cm,
   on grid,
   auto,
   every state/.style={draw=black!60, fill=black!5, very thick}
  ]
\node (a) {$v_1$};
\node (b)[right=of a] {$v_2$};
\node (c)[right= of b] {$v_3$};
\node (d)[right= of c] {$v_{n-2}$};
\node (e)[right= of d] {$v_{n-1}$};
\node (f)[right= of e] {$v_n$};
\draw[dotted] (c)--(d);
\path[->]
(a) edge node {$e_1$} (b)
(b) edge node {$e_2$} (c)
(d) edge node {$e_{n-1}$} (e)
(e) edge node {$e_{n-2}$} (f);
\end{tikzpicture}
\]
the lattice diagram is given by,\\
\[\begin{tikzpicture}[scale=.7]
  \node (u) at (-6,2) {$v_1$};
  \node (v) at (-1,2) {$v_2$};
\node (x) at (4,2) {$v_{n-2}$};
\node (y) at (6,2) {$v_{n-1}$};
\node (z) at (8,2) {$v_n$};
\node (a) at (-6,0) {$e_1{e_1}^*$};
\node (b) at (-1,0) {$e_2{e_2}^*$};
\node (aa) at (-6,-2){$(e_1e_2)(e_1e_2)^*$};
\node (aaa) at (-6,-5) {$(e_1...e_{n})(e_1...e_{n})^*$};
\node (bb) at (-1,-2) {$(e_2e_3)(e_2e_3)^*$};
\node (bbb) at (-1,-5) {$(e_2...e_{n})(e_2...e_{n})^*$};
\node (xx) at (4,0) {$e_{n-1}{e_{n-1}}^*$};
\node (xxx) at (4,-2) {$(e_{n-1}e_n)(e_{n-1}e_n)^*$};
\node (yy) at (6,0){$e_n{e_n}^*$};
  \node (zero) at (3.5,-8) {$0$};
\draw[dotted](v)--(x);
\draw[dotted](b)--(xx);
 \draw (u) -- (a) ;
\draw (v) -- (b) ;
\draw (a)--(aa);
\draw (b)--(bb);
\draw (x)--(xx);
\draw (xx)--(xxx);
\draw (xxx)--(zero);
\draw (y)--(yy);
\draw[dotted](aa)--(aaa);
\draw[dotted](bb)--(bbb);
\draw(aaa)--(zero);
\draw (bbb)--(zero);
\draw (yy)--(zero);
\draw (z)--(zero);
\end{tikzpicture}
\]
\begin{dfn}(cf. \cite{lawson})
A non-zero idempotent of a semigroup with zero is said to be primitive if it is minimal relative to the natural partial order on the set of non-zero idempotents.
\end{dfn}

 An inverse semigroup with zero is said to be primitive if every nonzero idempotent is primitive. 

\begin{thm}
A graph inverse semigroup $I(\Gamma)$ is primitive if and only if $\Gamma$ is totally a disconnected or a trivial  graph. 
\end{thm}

\begin{proof}
Suppose $I(\Gamma)$ is a primitive inverse semigroup. Since maximal idempotents of $I(\Gamma)$ are vertices of 
$\Gamma$ and among non-zero idempotents minimal elements are ${p_i}{p_i}^*$, where $p_i$ is a longest path in $\Gamma$ starting from some vertex $v_i$. If $e$ is an edge starting from $v$ then $ee^*\,\leq \,v$, hence $v$ cannot be a minimal element and so a vertex $v$ is minimal among the set of non-zero idempotents if and only if there exist no edge $e$ starting from $v$. Since each non-zero idempotents are primitive it is not possible to have an edge $e$ starting from any vertex of $\Gamma$. Therefore $\Gamma$ does not have any edge, ie., $\Gamma$ is either totally disconnected or trivial.\\
Conversely, assume that $\Gamma$ is totally disconnected or trivial. Then none of the vertices in $\Gamma$ are comparable under the natural partial order. Hence $I(\Gamma)$ is a primitive inverse semigroup.
\end{proof}

\subsection{Local submonoid}
The local submonoid $pp^*I(\Gamma)pp^*$ is an inverse submonoid of $I(\Gamma)$ having identity $pp^*$. General form of an element in $pp^*I(\Gamma)pp^*$ is given by,\\
 \[\begin{array}{lcl}
\mbox{$(pp^*)(rs^*)(pp^*)$} & = &{\left\{ \begin{array}{ll}
          pp^*(p\mu)(p\nu)^*pp^* & \mbox{if $\exists\,\mu,\nu\,\ni\, r=p\mu$ and $s=p\nu$}\\
          0  & \mbox{otherwise};
          \end{array} \right. }\\
\mbox{} & = &{\left\{ \begin{array}{ll}
          (p\mu)(p\nu)^* & \mbox{if $\exists\,\mu,\nu\,\ni\, r=p\mu$ and $s=p\nu$}\\
          0  & \mbox{otherwise};
          \end{array} \right. }
\end{array} 
\]
hence,\\
\[
\begin{split} pp^*I(\Gamma)pp^*\,=\bigg\{0,(p\mu)(p\nu)^*\mid\mu\,\text{and}\,\nu\,\text{are directed paths 
with}\,r(\mu)=r(\nu)\\ and\, s(\mu)=s(\nu)=r(p)\,\bigg\}\end{split}
\]
if $pp^*\,=\,v$ for some vertex $v$, then,\\
\[
 vI(\Gamma)v\,=\bigg\{0,\mu{\nu}^*\mid \,r(\mu)=r(\nu) and\, s(\mu)=s(\nu)=v\,\bigg\}
\]
ie., the local submonoid $vI(\Gamma)v$ contains all circutes at $v$ of the form $pq^*$, where $p$ and $q$ are directed paths starting at $v$ and ending at the same vertex. If the choosen vertex $v$ is a sink, then 
the trivial path $v$ is the only path starting from $v$ and so $vI(\Gamma)v\,=\,\{0,\,v\} $.

Conversely let $ vI(\Gamma)v\,=\,\{0,\,v\} $. If there exists an edge $e$ starting from $v$, then $ee^*$ will be in $ vI(\Gamma)v$. But, $ vI(\Gamma)v$ is trivial and so it is not possible to have an edge starting from $v$. Hence the local submonoid at $v$ is trivial if and only if $v$ is a sink. For a totally disconnected graph $\Gamma$ each vertex is a sink and so $vI(\Gamma)v$ is trivial for every $v\,\in\Gamma^0$. Hence a directed graph $\Gamma$ is totally disconnected if and only if local submonoid of $I(\Gamma)$ at each  vertex $v$  is trivial. These observation can be formulated as follows,
\begin{thm} Let $\Gamma$ be a directed graph and $I(\Gamma)$ the graph inverse semigroup. Then the local submonoid $vI(\Gamma)v$ for a vertex $v$ of $\Gamma$ is trivial if and only if $v$ is a sink.
\end{thm}

\subsection{Green's relations on $I(\Gamma)$}

\begin{enumerate}
\item $pq^* \,\mathscr{L}\, xy^*\,$  if and only if  $ \, q\,=\,y$\\
For elements $s,t$ in an inverse semigroup, $s\,\mathscr{L}\, t\,\Leftrightarrow \, s^{-1}s\,=\,t^{-1}t$. So,
\[\begin{array}{lcl}
\mbox{$pq^* \,\mathscr{L}\, xy^*$} & \Leftrightarrow & (pq^*)^*(pq^*)\,=\,(xy^*)^*(xy^*)\\
\mbox{} & \Leftrightarrow & qp^*pq*\,=\,yx^*xy^*\\
\mbox{} & \Leftrightarrow & qq^*\,=\, yy^*\\
\mbox{} & \Leftrightarrow & q\,=\,y
\end{array}\]

\item  $pq^* \,\mathscr{R}\, xy^*\,$  if and only if  $ \, p\,=\,x$\\
For elements $s,t$ in an inverse semigroup, $s\,\mathscr{R}\, t\,\Leftrightarrow \,s s^{-1}\,=\,tt^{-1}$. Therefore,
\[\begin{array}{lcl}
\mbox{$pq^* \,\mathscr{R}\, xy^*$} & \Leftrightarrow & (pq^*)(pq^*)^*\,=\,(xy^*)(xy^*)^*\\
\mbox{} & \Leftrightarrow & pq*qp^*\,=\,xy^*yx^*\\
\mbox{} & \Leftrightarrow & pp^*\,=\, xx^*\\
\mbox{} & \Leftrightarrow & p\,=\,x
\end{array}\]

\item  For elements $pq^*,\,xy^*$ in $I(\Gamma),\,\,pq^* \,\mathscr{D}\, xy^*$ if and only if $r(p)\,=\,r(x)$\\
\[\begin{array}{lcl}
\mbox{$pq^* \,\mathscr{D}\, xy^*$} & \Leftrightarrow & \exists\, rs^*\,\in\, I(\Gamma) \ni pq^*\mathscr{L}rs^*\mathscr{R}xy^*\\
\mbox{} & \Leftrightarrow & s\,=\,q \, and \,r\,=\, x\\
\mbox{} & \Leftrightarrow & r(p)\,=\,r(x)\\
\end{array}\]

\item $pq^*\mathscr{J}xy^* $  if and only if  $\, r(p)$ and $r(x)$ are in the same strongly connected component of $\Gamma$.

If $pq^*\mathscr{J}xy^*$ then $pq^*\,\in\, I(\Gamma)xy^*I(\Gamma)$ and 
$$xy^*\,=\,xx^*xy^*\,=\, xr(x)y^*\,\in\, I(\Gamma)r(x)I(\Gamma).$$ 
Also, $r(x)\,=\,r(x)r(x)\,=\, xx^*yy^*\,=\,x(xy^*)y^*\,\in\, I(\Gamma)xy^*I(\Gamma)$ and hence $ I(\Gamma)xy^*I(\Gamma)\,=\, I(\Gamma)r(x)I(\Gamma)$. ie., 
$$pq^*\,\in  I(\Gamma)pq^*I(\Gamma) \,\subseteq  I(\Gamma)r(x)I(\Gamma).$$
Thus $r(p)\,=p^*(pq^*)q\,\in I(\Gamma)r(x)I(\Gamma)$ which implies $r(p)\,=\,p_1{q_1}^*r(x)p_2{q_2}^*$ 
where $s(q_1)=r(x)=s(p_2),\, r(p_1)=r(q_1)$ and $r(p_2)=r(q_2)$. ie., $r(p)\,=\, p_1{q_1}^*p_2{q_2}^*$ and 
$r(p)$ is a vertex in $\Gamma$. Hence $p_1,q_2\,\in \Gamma^0$ and $q_1\,=\, p_2$. 
Therefore $p_1\,=\,q_2\,=\,r(p)$ and $q_1$ is a directed path from $r(p)$ to $r(x)$. Using same arguments we get a directed path from $r(x)$ to $r(p)$. Hence $r(p)$ and $r(x)$ are in the same strongly connected component

Conversely, if $r(x)$ and $r(p)$ are in the same strongly connected component, then there exists directed paths 
$t\, and\, t'$ from $r(x)$ to $r(p)$ and $r(p)$ to $r(x)$ respectively. Then $r(p)= t^*t=t^*r(x)t\,\in  I(\Gamma)r(x)I(\Gamma)=I(\Gamma)xy^*I(\Gamma)$ hence $pq^*\,=\,pp^*pq^*\,=\,pr(p)q^*\,\in\,I(\Gamma)xy^*I(\Gamma)$ and so 
$I(\Gamma)pq^*I(\Gamma)\, \subseteq \,I(\Gamma)xy^*I(\Gamma)$. Similarly we get $I(\Gamma)xy^*I(\Gamma)\, \subseteq\,I(\Gamma)pq^*I(\Gamma)$. Hence $I(\Gamma)xy^*I(\Gamma)\,=\,I(\Gamma)pq^*I(\Gamma)$ ie., $ pq^*\mathscr{J}xy^*$.

\item $pq^*\mathscr{H}xy^*\,$  if and only if  $pq^*\,=\, xy^*$\\
\[
\begin{array}{lcl}
\mbox{$pq^*\mathscr{H}xy^*$} & \Leftrightarrow & pq^* \,\mathscr{L}\, xy^* \,and\,pq^* \,\mathscr{R}\, xy^*\\
\mbox{} & \Leftrightarrow & q\,=\,y \, and\, p\,=\,x\\
\mbox{} & \Leftrightarrow &  pq^*\,=\, xy^*
\end{array}
\]
\end{enumerate}
ie., $\mathscr{H}\,=\,\mathcal{I}_{I(\Gamma)}$. Hence for any directed graph $\Gamma$, $I(\Gamma)$  is combinatorial.

\begin{exam}
Consider $P_2$\\
\[
\begin{tikzpicture}[>=stealth,
   shorten >=1pt,
   node distance=2cm,
   on grid,
   auto,
   every state/.style={draw=black!60, fill=black!5, very thick}
 ]
\node (mid)        {u};
\node (right)[right=of mid] {v};
\path[->]
%   FROM       BEND/LOOP           POSITION OF LABEL   LABEL   TO
   (mid)  edge             node                      {e} (right)
   ;
\end{tikzpicture}
\]
then, $I(P_2)\,=\,\{u,v,e,e^*,ee^*,0\}$ and the Green's relations are\\
$\mathscr{L}\,=\,\{(v,e),(e,v),(e^*,ee^*),(ee^*,e^*),\}\cup\mathcal{I}_{I(P_2)}$\\
$\mathscr{R}\,=\,\{(u,e^*),(e^*,u),(e,ee^*),(ee^*,e)\}\cup\mathcal{I}_{p_2}$\\
$\mathscr{H}\,=\,\mathcal{I}_{P_2}$\\
\[\begin{split}\mathscr{D}\,=\,\{(v,e),(e,v),(v,e^*),(e^*,v),(v,ee^*),(ee^*,v),(e,e^*),(e^*,e),(e,ee^*),\\(ee^*,e),(e^*,ee^*),(ee^*,e^*)\}\cup\mathcal{I}_{I(P_2)}
\end{split}
\]
\[\begin{split}\mathscr{J}\,=\,\{(v,e),(e,v),(v,e^*),(e^*,v),(v,ee^*),(ee^*,v),(e,e^*),(e^*,e),(e,ee^*),\\(ee^*,e),(e^*,ee^*),(ee^*,e^*)\}\cup\mathcal{I}_{I(P_2)}
\end{split}
\]
\end{exam}

\end{document}